\documentclass[12pt]{amsart}
\usepackage{amsfonts}
\usepackage{amsmath}
\usepackage{amssymb}
\usepackage{texdraw}
\usepackage{latexcad}

\newtheorem{theorem}{Theorem}[section]
\theoremstyle{plain}

\newtheorem{corollary}[theorem]{Corollary}

\newtheorem{definition}[theorem]{Definition}

\newtheorem{lemma}[theorem]{Lemma}
\newtheorem{notation}[theorem]{Notation}
\newtheorem{problem}[theorem]{Problem}
\newtheorem{proposition}[theorem]{Proposition}

\numberwithin{equation}{section}

\begin{document}

\title[Colored Graphs without Colorful Cycles]{Colored Graphs \\without Colorful Cycles}

\author{Richard N. Ball}
\address[Ball]{Dept. of Mathematics,
University of Denver, 2360 S Gaylord St., Denver, Colorado 80208, U.S.A.}
\email{rick.ball@nsm.du.edu} \urladdr{www.math.du.edu/\symbol{126}rball}

\author{Ale\v{s} Pultr}
\address[Pultr]{Dept. of Applied Mathematics and ITI, MFF
Charles University, Malostransk\'{e} n\'{a}m. 25, Praha 1, 11800, Czech
Republic} \email{pultr@kam.ms.mff.cuni.cz}

\author{Petr Vojt\v{e}chovsk\'{y}}
\address[Vojt\v{e}chovsk\'y]{Dept. of Mathematics,
University of Denver, 2360 S Gaylord St., Denver, Colorado 80208, U.S.A.}
\email{petr@math.du.edu} \urladdr{www.math.du.edu/\symbol{126}petr}

\thanks{The first author would like to express his thanks for support by project
1M0021620808 of the Ministry of Education of the Czech Republic.}

\thanks{The second author would like to express his thanks for support by project
1M0021620808 of the Ministry of Education of the Czech Republic, by the NSERC
of Canada and by the Gudder Trust of the University of Denver.}





\begin{abstract}
A colored graph is a complete graph in which a color has been assigned to each
edge, and a colorful cycle is a cycle in which each edge has a different color.
We first show that a colored graph lacks colorful cycles iff it is Gallai,
i.e., lacks colorful triangles. We then show that, under the operation $m\circ
n\equiv m+n-2$, the omitted lengths of colorful cycles in a colored graph form
a monoid isomorphic to a submonoid of the natural numbers which contains all
integers past some point. We prove that several but not all such monoids are
realized.

We then characterize exact Gallai graphs, i.e., graphs in which every triangle
has edges of exactly two colors. We show that these are precisely the graphs
which can be iteratively built up from three simple colored graphs, having $2$,
$4$, and $5$ vertices, respectively. We then characterize in two different ways
the monochromes, i.e., the connected components of maximal monochromatic
subgraphs, of exact Gallai graphs. The first characterization is in terms of
their reduced form, a notion which hinges on the important idea of a full
homomorphism. The second characterization is by means of a homomorphism
duality.
\end{abstract}

\maketitle

\section{Introduction}

For the purposes of constructing coproducts of distributive lattices, the
first two authors found certain edge-colorings of complete graphs to be
useful. The specific colorings of use were those lacking colorful cycles of
particular lengths. It turns out that such colorings exhibit a structure which
may be of interest in its own right. We investigate that structure here.

The absence of short colorful cycles implies the absence of certain longer
ones, and this fact leads to the concept of the spectrum, defined and analyzed
in Section \ref{Sec:Spec}. Gallai colorings, i.e., colorings which lack
colorful 3-cycles, constitute an extreme example of this phenomenon, for they
have no colorful cycles at all (Proposition \ref{Pr:3=All}).

We therefore turn our attention to Gallai colorings. These colorings are known
to have a simple and pleasing structure, which we review and elaborate in
Section \ref{Sec:GallCliq}. We then impose the additional hypothesis of
exactness, i.e., the hypothesis that every 3-cycle has edges of exactly two
colors. The resulting structural description, given in Section
\ref{Sec:ExGallCliq}, is especially sharp, and, in fact, the analysis can be
considered to be complete. The structural results make it possible to
characterize, in Section \ref{Sec:FullHom}, the monochromes, i.e., the
components of the monochromatic subgraphs. This section introduces the
important notion of a full homomorphism, which is then used in Section
\ref{Sec:HomDual} to elaborate the characterization of exact Gallai monochromes
by means of a homomorphism duality.

Gallai initiated the investigation of the colored graphs which now bear his
name in his foundational paper \cite{Gallai:1967}. Since then, these graphs
have appeared in several different contexts and for different reasons. We
mention only two of the more recent investigations: Gy\'{a}rf\'{a}s and Simonyi
showed the existence of monochromatic spanning brushes in
\cite{GyarfasSimonyi:2004}; Chung and Graham found the bound on the maximum
number of vertices for a given number of colors in exact Gallai cliques in
\cite{ChungGraham:1983} (see Theorem \ref{Thm:VerBnd}). A good general
background reference is the survey article \cite{Nesetril:1995}.

\section{Ground clearing}

Graphs will be assumed to be finite, symmetric (undirected), and without loops.
We denote a graph $G$ by $\left(  V_{G},E_{G}\right)  $, where $V_{G}$ and
$E_{G}$ designate the sets of vertices and edges of $G$, respectively. Symbols
$u$, $v$, and $w$ are reserved for vertices, with the edge connecting vertices
$u$ and $v$ designated by $uv$. The symbol $K\ $is reserved for complete
graphs.

An \emph{edge coloring of a graph }$G$, or simply a \emph{coloring of }$G$, is
an assignment of an element of a finite set $\Gamma$ of \emph{colors} to each
edge of $G$. We use lower-case Greek letters to designate the individual
colors, upper-case Greek letters to designate sets of colors, $\overline{uv}$
to designate the color assigned to the edge $uv$, and $\overline{\bullet}$ to
designate the coloring map itself. A \emph{colored graph} is an object of the
form $G=\left(  V,E,\overline{\bullet}\right)  $, where $\left(  V,E\right)  $
is a graph and $\overline{\bullet}:E\rightarrow \Gamma$ is a coloring.

In any graph, a \emph{clique} is a complete subgraph induced by a nonempty
subset of vertices. We denote cliques by symbols $a$, $b$, $c$, $d$, and for
cliques $a$ and $b$, we denote the set of edges joining their vertices by
\[
ab\equiv\left\{  uv\in E:u\in a,\ v\in b\right\}  .
\]
In most instances, our graphs will be complete, so that the cliques could be
identified with the corresponding vertex subsets. Still, we prefer to speak of
cliques to emphasize that we deal with edges rather than with vertices.

In a colored graph $G$, a clique is regarded as a colored graph under the
restriction of the coloring of $G$. For cliques $a$ and $b$, we designate by
$\overline{ab}$ the set of colors of the edges of $ab$; note that
$\overline{aa}=\emptyset$ if $\left\vert a\right\vert =1$.

In any graph, an $n$\emph{-cycle}, $n\geq2$, is a sequence $\left(
v_{1},v_{2},\ldots,v_{n}\right)  $ of distinct vertices. Two cycles are
regarded as identical if they can be made to coincide by a cyclical
($v_{i}\longmapsto v_{j+i}$) permutation of their elements, where all
subscript arithmetic is performed $\operatorname{mod}n$. $3$-cycles are called
\emph{triangles}, $4$-cycles are called \emph{squares}, and so forth. The
\emph{edges of a cycle }$\left(  v_{1},v_{2},\ldots,v_{n}\right)  $ are those
of the form $v_{i}v_{i+1}$.

In a colored graph, a cycle $\left(  v_{1},v_{2},\ldots,v_{n}\right)  $ is
\emph{colorful} if all its edges have different colors, i.e., if
\begin{displaymath}
    \overline{v_{i}v_{i+1}}=\overline{v_{j}v_{j+1}}\Longleftrightarrow
    i=j\operatorname{mod}n.
\end{displaymath}
Note that a $2$-cycle is never colorful. A \emph{Gallai clique} is a clique
which is complete and has no colorful triangles. An \emph{exact Gallai clique}
is a Gallai clique in which every triangle has edges of exactly two colors.

\section{The spectrum of a colored graph}\label{Sec:Spec}

In complete colored graphs, the absence of colorful cycles of a particular
length implies the absence of certain longer colorful cycles. In particular,
the absence of colorful triangles implies the absence of colorful cycles of any
length. In this section we prove this fact (Proposition \ref{Pr:3=All}) and
more. \emph{The running assumption throughout this section is that we are
dealing with a complete colored graph }$K\emph{.}$\emph{ }

The \emph{spectrum} of a coloring $\overline{\bullet}$ is the set of
prohibited lengths of colorful cycles, designated
\begin{displaymath}
    S\left(  \overline{\bullet}\right)  \equiv\left\{  n\geq2:\text{there are no
    colorful }n\text{-cycles}\right\}  .
\end{displaymath}
Obviously, every spectrum contains $2$, and contains all integers
$n>\left\vert K\right\vert $. The set of all spectra will be denoted by
$\mathcal{S}$.

On the set $\{2,3,\dots\}$ define an operation $\circ$ by setting
\begin{displaymath}
    m\circ n=m+n-2.
\end{displaymath}
The monoid so obtained is isomorphic to the additive monoid
$\mathbb{N}=\{0,1,\dots\}$ of natural numbers via $n\mapsto n-2$; we denote it
$\mathbb{N}(2)$.

\begin{proposition}
Every spectrum $S\in\mathcal{S}$ is a submonoid of $\mathbb{N}(2)$ which is
\emph{eventually solid}, i.e., contains all integers $k\geq n$ for some $n$.
\end{proposition}

\begin{proof}
Suppose that $m$, $n\in S(\overline{\bullet})$. Let $C$ be an $(m+n-2)$-cycle.
There is a chord of $C$ that makes $C$ into an $m$-cycle conjoined with an
$n$-cycle along the chord. Since $m\in S\left( \overline{\bullet}\right)  $,
the color of the chord must match the color of some other edge from the
$m$-cycle, and likewise that of some other edge of the $n$-cycle. This means
that $C$ is not colorful.
\end{proof}

Since $3$ is the unique generator of $\mathbb{N}(2)$ corresponding to $1$ in
$\mathbb{N}$, we obtain the following insight.

\begin{proposition}
\label{Pr:3=All}If $3\in S\in\mathcal{S}$ then $S$ is the whole of
$\mathbb{N}(2)$. In other words, if a colored graph contains no colorful
triangles then it contains no colorful cycles at all.
\end{proposition}

\begin{proposition}
\label{Pr:4S} Assume that $S\in\mathcal{S}$ satisfies $4\in S$. Then there is
an odd integer $m\geq3$ such that
\begin{displaymath}
    S=\{2,4,6,8,\dots,m-1,m,m+1,m+2,\dots\}.
\end{displaymath}
\end{proposition}

\begin{proof}
The submonoid of $\mathbb{N}(2)$ generated by $4$ consists of all positive
even integers. Let $m$ be the smallest positive odd integer in $S$. Then
$m+2=4+m-2=4\circ m\in S$, $m+4=4\circ(m+2)\in S$, and so forth.
\end{proof}

\begin{corollary}
\label{Cr:4S} If a colored graph has no colorful squares and no colorful
pentagons then it has no colorful $n$-cycles for any $n>3$.
\end{corollary}

The simplest question suggested by Proposition \ref{Pr:4S} is whether the
integer $m$ mentioned there can be any odd number, i.e., whether colored
graphs without colorful squares can admit colorful $m$-gons for arbitrary odd
integers $m$. The answer to this question is positive.

\begin{proposition}
\label{Pr:4S2}For every odd integer $m>1$ there is a colored graph with $m$
vertices having a colorful $m$-gon but no colorful squares.
\end{proposition}

\begin{proof}
Let $m=2k+1$, and label the vertices $v_{i}$, $-k\leq i\leq k$. We employ a
palette consisting of distinct colors $\alpha$ and $\beta_{i}$, $-k\leq i\leq
k$, $i\neq0$. For distinct indices $i$ and $j$, set
\begin{displaymath}
    \overline{v_{i}v_{j}}\equiv\left\{
    \begin{array}
    [c]{lll}
    \alpha & \text{if} & i\text{ and }j\text{ have the same parity,}\\
    \beta_{i} & \text{if} & i\text{ and }j\text{ have different parity and
    }\left\vert i\right\vert >\left\vert j\right\vert .
    \end{array}
    \right.
\end{displaymath}
The cycle $\left(  v_{-k},v_{-k+1},\ldots,v_{k}\right)  $ is colorful, with
the color of the edges in order being
\begin{displaymath}
    \beta_{-k},\beta_{-k+1},\ldots,\beta_{-1},\beta_{1},\ldots,\beta_{k-1},
    \beta_{k},\alpha.
\end{displaymath}

It remains to show that there are no colorful squares. Let $\left(
v_{i},v_{j},v_{k},v_{l}\right)  $ be a square. Assume that the following
happens at least twice around the square:
\begin{equation}
\text{Two consecutive vertices have the same parity.} \tag{$\ast$}
\end{equation}
Then at least two of the four edges are colored by $\alpha$, and the square is
not colorful. We can therefore assume that ($\ast$) happens at most once. But
($\ast$) cannot happen precisely once since the square has $4$ vertices, and so
($\ast$) never happens. Without loss of generality, let $i$ have the maximum
absolute value among the four indices. Since ($\ast$) never happens, we
conclude that $|j|<|i|$ and $|k|<|i|$. But then
$\overline{v_{i}v_{j}}=\overline{v_{i}v_{k}}=\beta_{i}$.
\end{proof}

\setlength{\unitlength}{1.2mm}
\begin{figure}[th]
\begin{center}
\begin{small}
\input{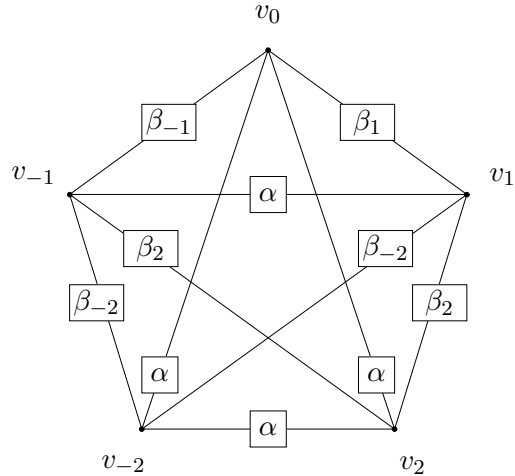}
\end{small}
\end{center}
\caption{A colorful pentagon without colorful squares.} \label{Fg:Pentagon}
\end{figure}

Figure \ref{Fg:Pentagon} shows a colorful pentagon without colorful squares
obtained by the construction of Proposition \ref{Pr:4S2}.

We have shown that if $S$ is an eventually solid submonoid of $\mathbb{N}(2)$
containing $3$ or $4$ then $S\in\mathcal{S}$. An older version of this paper
asked the question whether every eventually solid submonoid of $\mathbb{N}(2)$
can be found in $\mathcal{S}$. We now know that this is not the case, thanks to
results of Boris Alexeev \cite{Alexeev:2005}, who observed that a decagon with
no colorful pentagons is itself not colorful, and proved that a colored graph
with no colorful $n$-gons, for some $n>1$ odd, contains no colorful cycles of
length greater than $2n^2$. See \cite{Alexeev:2005} for more details. Here is
our proof of Alexeev's observation:

\begin{lemma} There is no colorful decagon without colorful pentagons.
\end{lemma}
\begin{proof}
Suppose for a contradiction that $v_0v_1\cdots v_9$ is a colorful decagon
without colorful pentagons. Let $\alpha_{ij}$ be the color of $v_iv_j$, and set
$\alpha_{01}=\gamma_0$, $\alpha_{12}=\gamma_1$, $\dots$,
$\alpha_{89}=\gamma_8$, $\alpha_{90}=\gamma_9$, where $\gamma_0$, $\dots$,
$\gamma_9$ are distinct colors.

Since the pentagon $v_0v_1v_2v_3v_4$ is not colorful, we must have
$\alpha_{04}\in\{\gamma_0$, $\gamma_1$, $\gamma_2$, $\gamma_3\}$. Similarly,
the pentagons $v_0v_6v_7v_8v_9$, $v_1v_2v_3v_4v_5$, $v_1v_7v_8v_9v_0$ and
$v_2v_3v_4v_5v_6$ show that $\alpha_{06}\in\{\gamma_6$, $\gamma_7$, $\gamma_8$,
$\gamma_9\}$, $\alpha_{15}\in\{\gamma_1$, $\gamma_2$, $\gamma_3$, $\gamma_4\}$,
$\alpha_{17}\in\{\gamma_0$, $\gamma_7$, $\gamma_8$, $\gamma_9\}$, and
$\alpha_{26}\in\{\gamma_2$, $\gamma_3$, $\gamma_4$, $\gamma_5\}$, respectively.
The pentagon $v_2v_6v_0v_4v_3$ then implies that $\alpha_{26}\in\{\gamma_0$,
$\gamma_1$, $\gamma_2$, $\gamma_3$, $\gamma_6$, $\gamma_7$, $\gamma_8$,
$\gamma_9\}$, and hence $\alpha_{26}\in\{\gamma_2$, $\gamma_3\}$.

We finish the proof by eliminating all possible colors for $\alpha_{05}$. The
pentagon $v_0v_5v_6v_7v_1$ shows that $\alpha_{05}\in\{\gamma_0$, $\gamma_5$,
$\gamma_6$, $\gamma_7$, $\gamma_8$, $\gamma_9\}$, and $v_0v_5v_6v_2v_1$ implies
$\alpha_{05}\in\{\gamma_0$, $\gamma_1$, $\gamma_2$, $\gamma_3$, $\gamma_5\}$.
Finally, $v_0v_5v_1v_2v_6$ yields $\alpha_{05}\in\{\gamma_1$, $\gamma_2$,
$\gamma_3$, $\gamma_4$, $\gamma_6$, $\gamma_7$, $\gamma_8$, $\gamma_9\}$, and
we are through.
\end{proof}

We close this section with another construction in the positive direction:

\begin{proposition}
\label{Pr:MinusOne} For every $m>2$ there is a colorful $2m$-gon without
colorful $(2m-1)$-gons.
\end{proposition}

\setlength{\unitlength}{1.3mm}
\begin{figure}[th]
\begin{center}
\begin{small}
\input{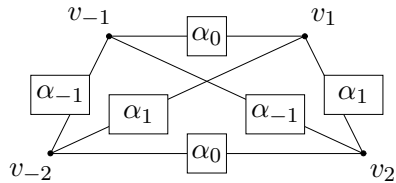}
\end{small}
\end{center}
\caption{Construction of Proposition \ref{Pr:MinusOne}.}\label{Fg:Clique}
\end{figure}

\begin{proof}
Draw the complete graph $K$ on $2m$ vertices in the usual way, as a convex
$2m$-gon $P$ on the perimeter and all remaining inner edges as straight line
segments inside $P$. We say that two inner edges \emph{cross} if they have a
point in common that is not a vertex of $K$. Color $P$ by $2m$ distinct colors.
Pick four consecutive vertices on $P$, say $v_{-2}$, $v_{-1}$, $v_{1}$,
$v_{2}$, and assume that $\overline{v_{-2}v_{-1}}=\alpha_{-1}$,
$\overline{v_{-1}v_{1}}=\alpha_{0}$, $\overline{v_{1}v_{2}}=\alpha_{1}$. Color
inner edges as follows: $\overline{v_{-2}v_{1}}=\alpha_{1}$, $\overline
{v_{-1}v_{2}}=\alpha_{-1}$, all remaining inner edges are colored $\alpha_{0}$.
The clique $\{v_{-2}$, $v_{-1}$, $v_{1}$, $v_{2}\}$ is depicted in Figure
\ref{Fg:Clique}.

Let $H$ be a colorful $(2m-1)$-gon in the above coloring, and let $n$ be the
number of crossings among the edges of $H$. If $n=0$, then $H$ lies on $P$ with
the exception of one edge $e$ that skips a vertex on $P$. If $e$ skips $v_{-1}$
then $H$ has two edges colored $\alpha_{1}$, namely $v_{-2}v_{1}$ and
$v_{1}v_{2}$. If $e$ skips $v_{1}$ then $H$ has two edges colored $\alpha
_{-1}$, namely $v_{-2}v_{-1}$ and $v_{-1}v_{2}$. If $e$ skips any other vertex,
then $H$ has two edges colored $\alpha_{0}$, namely $e$ and $v_{-1}v_{1}$.

We claim that $v_{-2}v_{1}$ and $v_{-1}v_{2}$ cannot both lie in $H$. Assume
they do. If $v_{-1}v_{1}$ is also in $H$, then $H$ has two edges colored
$\alpha_{0}$, since it must have another inner edge besides $v_{-2}v_{1}$,
$v_{-1}v_{2}$. So $v_{-1}v_{1}$ is not in $H$. Since
$\overline{v_{-2}v_{-1}}=\alpha_{-1}$, $H$ must continue from $v_{-1}$ via some
inner edge colored $\alpha_{0}$. Since $\overline{v_{1}v_{2}}=\alpha_{1}$, $H$
must continue from $v_{1}$ via some inner edge colored $\alpha_{0}$, a
contradiction. We have proved the claim.

Assume $n=1$. Since the crossing edges of $H$ have distinct colors, say
$\alpha$ and $\beta$, either the color $\alpha$ is $\alpha_{1}$ or
$\alpha_{-1}$. There are therefore three scenarios: (i) $\alpha=\alpha_{1}$
and $\beta=\alpha_{-1}$. Then both $v_{-2}v_{1}$, $v_{-1}v_{2}$ are in $H$,
contrary to the claim. (ii) $\alpha=\alpha_{1}$ and $\beta=\alpha_{0}$. Then
$v_{-2}v_{1}$ is in $H$. But then $H$ cannot continue from $v_{1}$, since all
edges containing $v_{1}$ are colored $\alpha_{1}$ or $\alpha_{0}$. (iii)
$\alpha=\alpha_{-1}$ and $\beta=\alpha_{0}$. Then we are in a situation dual
to (ii).

Assume $n\ge 2$. Then $H$ has at least three inner edges, since two
edges only cross once. Hence all three colors $\alpha_{-1}$, $\alpha_{0}$,
$\alpha_{1}$ must be assigned to inner edges of $H$, and we have once again
violated the claim.
\end{proof}

\begin{problem} Characterize $\mathcal{S}$, the set of all spectra of complete
colored graphs.
\end{problem}


\section{Gallai cliques\label{Sec:GallCliq}}

The basic building blocks of Gallai cliques are the $2$-cliques, i.e., cliques
$a$ such that $\left\vert \overline{aa}\right\vert \leq2$, for a clique is
Gallai iff it can be iteratively built up from $2$-cliques. We flesh out this
result in Theorem \ref{4}, in more detail than would be strictly necessary if
that theorem were our only purpose. But the additional detail, and in
particular the concept of factor clique, is necessary for the subsequent
analysis of exact Gallai cliques in the following sections.

The fact that Gallai cliques can be iteratively built up from $2$-cliques
follows from Theorem \ref{Thm:H2C=Gall}. Following \cite{GyarfasSimonyi:2004},
we attribute this result to Gallai, for it is implicit in \cite{Gallai:1967}.
This theorem can also be found among the results of Cameron and Edmonds in
\cite{CameronEdmonds:1997}, and a nice proof is in \cite{GyarfasSimonyi:2004}.

\begin{definition}
Let $a$ be a clique in a colored graph, and let $\Delta\subseteq\Gamma$. A
$\Delta$\emph{-relation on }$a$ is an equivalence relation $R\subseteq a\times
a$ such that for all $u,v\in a$,
\begin{displaymath}
    \left(  u,v\right)  \notin R\Longrightarrow\overline{uv}\in\Delta.
\end{displaymath}
A $2$\emph{-relation on }$a$ is a $\Delta$-relation on $a$ for some
$\Delta\subseteq\Gamma$ such that $\left\vert \Delta\right\vert \leq2$. A
$\Delta$-relation is said to be \emph{homogenous} if for all $u_{i}$,
$v_{i}\in a$,
\begin{displaymath}
    \left(  \left(  u_{1},u_{2}\right)  ,\left(  v_{1},v_{2}\right)  \in R\text{
    and }\left(  u_{1},v_{1}\right)  \notin R\right)  \Longrightarrow
    \overline{u_{1}v_{1}}=\overline{u_{2}v_{2}}.
\end{displaymath}
The descriptive adjectives of the relations apply to the partitions they
induce, giving the terms $\Delta$\emph{-partition}, $2$\emph{-partition}, and
\emph{homogenous partition}.
\end{definition}

To rephrase the definition, a $\Delta$-partition of $a$ is a pairwise disjoint
family $A$ of cliques whose union is $a$ and which satisfies
\begin{displaymath}
    \bigcup\left\{  \overline{a_{1}a_{2}}:a_{i}\in A,\ a_{1}\neq a_{2}\right\}
    \subseteq\Delta,
\end{displaymath}
and the relation is homogeneous if $\left\vert \overline{a_{1}a_{2}}\right\vert
=1$ for all $a_{i}\in A$ such that $a_{1}\neq a_{2}$.

\begin{theorem}
\label{Thm:H2C=Gall}A nonsingleton Gallai clique admits a nontrivial
homogeneous $2$-partition.
\end{theorem}

It is already clear from Theorem \ref{Thm:H2C=Gall} that Gallai cliques are
iteratively built up from $2$-cliques. What is necessary now is to identify,
conceptually and notationally, the particular $2$-cliques used in the
formation of a given Gallai clique. Thus we are led to the notions of
hereditary $2$-clique and of tree $2$-clique.

\begin{definition}
\label{7}We inductively define a \emph{hereditary }$2$\emph{-clique} as
follows. A singleton clique is a hereditary $2$-clique. If a clique admits a
homogeneous $2$-partition whose parts are hereditary $2$-cliques then the
clique itself is a hereditary $2$-clique.
\end{definition}

A \emph{tree} is a finite poset $T$ in which every pair of unrelated elements
has a common upper bound but no common lower bound. In such a poset we define
\begin{displaymath}
    s\prec t\Longleftrightarrow\left(  s<t\text{ and }\forall\ r\ \left(  s\leq
    r\leq t\Longrightarrow r=s\text{ or }r=t\right)  \right)  ,
\end{displaymath}
and we say that $t$\emph{\ is the parent of }$s$, and that $s$\emph{\ is a
child of }$t$. We say that $s$\emph{\ is an offspring of }$t$, and that
$t$\emph{\ is an ancestor of }$s$, if $s<t$. Elements $s$ and $t$ of $T$ are
said to be \emph{siblings} if they are unrelated but share a parent. Note that
every pair of unrelated elements are the offspring of siblings. A childless
element is called a \emph{leaf}, and the set of leaves is called the
\emph{yield} of the tree,
\begin{displaymath}
    K\left(  T\right)  \equiv\left\{  t:t\text{ is a leaf}\right\}  .
\end{displaymath}
The largest element of a tree is referred to as its \emph{root}, and the
\emph{height} of a tree is the length of a longest path from a leaf to the root.

With a given tree $T$ we associate two graphs. The \emph{sibling graph
}$S\left(  T\right)  $ has as vertices the elements of $T$ and as edges all
those of the form $st$, where $s$ and $t$ are siblings. The \emph{leaf graph
}$K\left(  T\right)  $ is the complete graph on the yield of $T$. An
\emph{edge coloring of }$S\left(  T\right)  $, or simply a \emph{coloring of
}$S\left(  T\right)  $, is an assignment of a color, denoted $\widehat{st}$,
to each edge $st$. We use $\widehat{\bullet}$ to denote the color map itself.
If $\widehat{\bullet}$ has the additional property that for every $t\in
T\smallsetminus K\left(  T\right)  $
\begin{displaymath}
    \left\vert \left\{  \widehat{rs}:r\text{ and }s\text{ are distinct children of
    }t\right\}  \right\vert \leq2,\;
\end{displaymath}
then we say that $\widehat{\bullet}$ is a $2$\emph{-coloring of }$S\left(
T\right)  $.

\begin{proposition}
\label{5}Let $T$ be a tree. Any coloring $\widehat{\bullet}$ of $S\left(
T\right)  $ gives rise to a coloring $\overline{\bullet}$ of $K\left( T\right)$
by the rule
\begin{displaymath}
    \overline{st}\equiv\widehat{uv},
\end{displaymath}
where $u$ and $v$ are the respective sibling ancestors of $s$ and $t$. Such a
coloring satisfies
\begin{displaymath}
    \overline{st}=\overline{rt}
\end{displaymath}
whenever $s$ and $r$ have a common ancestor unrelated to $t$, and any coloring
of $K\left(  T\right)  $ with this property arises by this rule from a
coloring of $S\left(  T\right)  $.
\end{proposition}

We refer to a clique $a$ as a \emph{tree clique} if there is some tree $T$ and
some coloring of $S\left(  T\right)  $ such that, when $K\left(  T\right)  $
is colored as in Proposition \ref{5}, $a$ is isomorphic to $K\left(  T\right)
$. This means that there is a bijection from $a$ onto the leaves of $T$ which
preserves the color of the edges. If the coloring of $S\left(  T\right)  $ is
a $2$-coloring, we refer to $a$ as a \emph{tree }$2$\emph{-clique}.

\begin{proposition}
\label{2}A tree $2$-clique is Gallai.
\end{proposition}

\begin{proof}
We induct on the height of the tree. Consider vertices $u_{i}$, $1\leq i\leq
3$, in $a=K\left(  T\right)  $, where the edges of $a$ derive their colors from
a $2$-coloring of $S\left(  T\right)  $ as in Proposition \ref{5}. Label the
root of $T$ as $t_{0}$, and its children $t_{1}$, $t_{2}$, \ldots, $t_{n}$. If
all three vertices are offspring of a single $t_{i}$, the triangle they form
lies in $V\left(  \downarrow\!t_{i}\right)  $, the tree $2$-clique of the
subtree rooted at $t_{i}$. Since this subtree has height less than that of $T$,
the triangle is not colorful by the induction hypothesis. If two of the
vertices, say $u_{1}$ and $u_{2}$, are offspring of one $t_{i}$, while the
third vertex $u_{3}$ is the offspring of another $t_{j}$, $i\neq j$, then
\begin{displaymath}
    \overline{u_{1}u_{3}}=\widehat{t_{i}t_{j}}=\overline{u_{2}u_{3}}.
\end{displaymath}
If all three vertices are offspring of distinct children, say $u_{i}\leq$
$t_{j_{i}}$ for distinct $j_{i}$, $1\leq i\leq3$, then because $S\left(
T\right)  $ carries a $2$-coloring,
\begin{displaymath}
    \left\vert \left\{  \overline{u_{i}u_{k}}:1\leq i\neq k\leq3\right\}
    \right\vert =\left\vert \left\{  \widehat{t_{j_{i}}t_{j_{k}}}:1\leq i\neq
    k\leq3\right\}  \right\vert \leq2.
\end{displaymath}
Thus in any case the triangle formed by the $u_{i}$s is not colorful.
\end{proof}

\begin{proposition}
\label{6}A clique $a$ is a tree $2$-clique iff it is a hereditary $2$-clique.
\end{proposition}

\begin{proof}
Given a hereditary $2$-clique $a$, we build its tree inductively. If $a$ is a
singleton, its tree consists of a single root node. If $a$ admits a homogeneous
$2$-partition into hereditary $2$-cliques $a_{1}$, $a_{2}$, $\ldots$, $a_{k}$,
then for each $i$ there is, by the inductive hypothesis applied to $a_{i}$, a
tree $T_{i}$ and a $2$-coloring of $S\left( T_{i}\right)  $ such that $a_{i}$
is isomorphic to $K\left(  T_{i}\right)  $. Denote the root of each $T_{i}$ by
$t_{i}$. Form the tree $T$ for $a$ by using a new root node $t_{0}$, by
declaring the children of $t_{0}$ to be the $t_{i}$s, and by coloring the
sibling edges of the root by the rule
\begin{displaymath}
    \widehat{t_{i}t_{j}}\equiv\overline{a_{i}a_{j}},\;i\neq j.
\end{displaymath}
The result is a $2$-coloring of $S\left(  T\right)  $ which provides a natural
isomorphism from $a$ onto $K\left(  T\right)  $.

Now let a tree $T$ be given, along with a $2$-coloring of $S\left(  T\right)
$ and the corresponding coloring of $K\left(  T\right)  $ as in Proposition
\ref{5}. We show by induction on the height of $T$ that $K\left(  T\right)  $
is a hereditary $2$-clique. If the height of $T$ is $0$ then $T$ consists of
the root alone, and $K\left(  T\right)  $ is a singleton and therefore a
hereditary $2$-clique. So suppose we have established the result for trees of
height at most $n$, and consider a tree $T$ of height $n+1$ with root $t_{0}$
and children of the root $t_{1}$, $t_{2}$, \ldots, $t_{k}$. Let $T_{i}$ be
$\downarrow\!t_{i}$, the subtree of $T$ rooted at $t_{i}$. Then $a_{i}\equiv
K\left(  T_{i}\right)  $ is a hereditary $2$-clique by the inductive
hypothesis, and the partition into the $a_{i}$s makes $a\equiv K\left(
T\right)  $ into a hereditary $2$-clique as well.
\end{proof}

\begin{corollary}
\label{12}A hereditary $2$-clique is Gallai.
\end{corollary}

The expression of a given hereditary $2$-clique as a tree $2$-clique is by no
means unique. However, every such expression can be maximally refined, and
this is the content of Proposition \ref{1}. This proposition will be required
for the analysis in Section \ref{Sec:ExGallCliq} of exact Gallai cliques.

\begin{definition}
When a clique $a$ is expressed as a tree clique $K\left(  T\right)  $, for
each $t\in T\smallsetminus V\left(  T\right)  $ we refer to the clique of
$S(T)$ of the form
\begin{displaymath}
    \left\{  s:s\prec t\right\}
\end{displaymath}
as \emph{the factor of }$a$\emph{ at }$t$. For $t_{1}<t_{2}$ in
$T\smallsetminus V\left(  T\right)  $, we say that the factor at $t_{2}$ is
\emph{higher} than the factor at $t_{1}$.
\end{definition}

\begin{definition}
A\ clique is said to be \emph{irreducible} if it admits no nontrivial
homogeneous partition. A clique is said to be a \emph{hereditarily irreducible
}$2$\emph{-clique} provided that it can be represented as a tree $2$-clique
with irreducible factors.
\end{definition}

\begin{proposition}
\label{1}Every hereditary $2$-clique is a hereditarily irreducible $2$-clique.
\end{proposition}

\begin{proof}
By a process of successive refinement, the cliques which arise in expressing a
given hereditary $2$-clique as a tree $2$-clique can be rendered irreducible.
Of course, the height of the tree typically increases.
\end{proof}

We summarize our results to this point.

\begin{theorem}
\label{4}The following are equivalent for a complete clique $a$ in a colored graph.

\begin{enumerate}
\item $a$ is Gallai, i.e., $a$ has no colorful triangles.

\item $a$ has no colorful cycles.

\item $a$ is a hereditary $2$-clique.

\item $a$ is a hereditarily irreducible $2$-clique.

\item $a$ is a tree $2$-clique.

\item For disjoint subcliques $b$ and $c$ of $a$,
\begin{displaymath}
    \left\vert \overline{bc}\smallsetminus\overline{bb}\right\vert \leq\left\vert
    c\right\vert .
\end{displaymath}

\item For any subclique $b$ of $a$,
\begin{displaymath}
    \left\vert \overline{bb}\right\vert \leq\left\vert b\right\vert -1.
\end{displaymath}

\end{enumerate}
\end{theorem}

\begin{proof}
The equivalence of (1) and (2) is Proposition \ref{Pr:3=All}, that of (3) and (4) is
Proposition \ref{1}, that of (3) and (5) is Proposition \ref{6}, the
implication from (3) to (1) is Corollary \ref{12}, and the implication from
(1) to (3) yields to a simple induction based on Theorem \ref{Thm:H2C=Gall}.
(6) implies (1) by taking $\left\vert b\right\vert =2$ and $\left\vert
c\right\vert =1$, and (1) implies (6) by a simple induction on $\left\vert
c\right\vert $. Finally, (7) implies (1) by taking $\left\vert b\right\vert
=3$, and (1) implies (7) by a simple induction on $\left\vert b\right\vert $
based on (6).
\end{proof}


\section{Exact Gallai cliques\label{Sec:ExGallCliq}}

Now we turn our attention to exact Gallai cliques, i.e., complete cliques in
which every triangle has edges of exactly two colors. Their analysis requires
consideration of the monochromatic subgraphs of a colored graph $G=\left(
V,E,\overline{\bullet}\right)  $. More explicitly, for each color $\alpha$ we
have the (uncolored) graph
\begin{displaymath}
    G\left(  \alpha\right)  \equiv\left(  V,\left\{  e\in E:\overline{e}
    =\alpha\right\}  \right)  .
\end{displaymath}
A subgraph of $G$ is called \emph{monochromatic} if it is a subgraph of
$G\left(  \alpha\right)  $ for some $\alpha$. A \emph{monochrome} of $G$ is a
component of one of the $G\left(  \alpha\right)  $s, i.e., a maximal
connected monochromatic subgraph of $G$, considered as an uncolored graph.

Although the monochromes in Gallai cliques can be as complicated as one wishes
(Proposition \ref{Pr:GallMon}), the monochromes in exact Gallai cliques are
fairly simple (Proposition \ref{Pr:ExGallMon}), and the monochromes of the
irreducible factors of exact Gallai cliques are simple indeed (Definition
\ref{Def:SimFacCliq} and Proposition \ref{Pr:SimFacCliq}).

A subgraph is said to \emph{span} a graph if every vertex of the graph is a
vertex of the subgraph. The following result appeared first in
\cite{Gallai:1967}. The simple proof below was suggested by the referee:

\begin{proposition}
\label{Pr:GallMon}A Gallai clique has a spanning monochrome, and every
connected graph is a spanning monochrome in a Gallai clique.
\end{proposition}

\begin{proof}
Let $a$ be a Gallai clique, and let $M$ be a monochrome of $a$ with largest
number of vertices. We claim that $M$ spans $a$. Suppose that this is not the
case, and let $v$ be a vertex outside of $M$. Consider the star $S$ centered at
$v$ with leaves consisting of all vertices of $M$. We can assume that the edges
of $M$ are colored blue in $a$. If any of the edges of $S$ are colored blue, we
obtain a monochrome with more vertices than $M$, a contradiction. If $S$ is
monochromatic, we reach the same contradiction. Hence, without loss of
generality, there are vertices $u$, $w$ of $M$ such that $vu$ is red and $vw$
is green. Since $M$ is connected, there is a path $u_1=u$, $u_2$, $\dots$,
$u_n=w$ in $M$ such that $u_iu_{i+1}$ is blue for every $1\le i<n$. Since
$vu_1$ is red, $u_1u_2$ is blue, and $vu_2$ is not blue, $vu_2$ must be red,
else $a$ is not Gallai. Proceeding in this fashion, we conclude that $vu_n$
must be red, a contradiction.
\end{proof}

\begin{corollary}
A clique is Gallai iff every subclique has a spanning monochrome.
\end{corollary}

\begin{proof}
A triangle is a subclique.
\end{proof}

We will need to refer to several specific uncolored graphs.

\setlength{\unitlength}{0.7mm}
\begin{figure}[th]
\begin{center}
\input{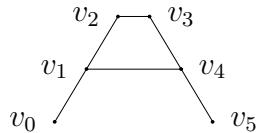}
\end{center}
\caption{The graph $A$.} \label{Fg:GraphA}
\end{figure}

\begin{notation}
\label{Note:SpGr}The $k$\emph{-path} is
\begin{displaymath}
    P_{k}\equiv\left(  \left\{  v_{i}:0\leq i\leq k\right\}  ,\left\{
    v_{i}v_{i+1}:0\leq i<k\right\}  \right)  ,
\end{displaymath}
and the $k$\emph{-cycle} is
\begin{displaymath}
    C_{k}\equiv\left(  \left\{  v_{i}:1\leq i\leq k\right\}  ,\left\{
    v_{i}v_{i+1}:1\leq i<k\right\}  \cup\left\{  v_{k}v_{1}\right\}  \right)  .
\end{displaymath}
We introduce a special graph which will play a role in Section
$\ref{Sec:HomDual}$:
\begin{displaymath}
    A\equiv\left(  \left\{  v_{i}:0\leq i\leq5\right\}  ,\left\{
    v_{0}v_{1},v_{1}v_{2},v_{1}v_{4},v_{2}v_{3},v_{3}v_{4},v_{4}v_{5}\right\}
    \right)  .
\end{displaymath}
See Figure $\ref{Fg:GraphA}$.
\end{notation}

\setlength{\unitlength}{1.2mm}
\begin{figure}[th]
\begin{center}
\input{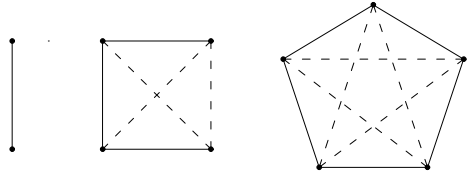}
\end{center}
\caption{Simple cliques.} \label{Fg:SimpleGallai}
\end{figure}

\begin{definition}
\label{Def:SimFacCliq}We say that a clique $a$ in a colored graph is
\emph{simple} if it is complete, and if either

\begin{enumerate}
\item $\left\vert a\right\vert =2$, or

\item $\left\vert a\right\vert =4$, and $a$ has two monochromes isomorphic to
$P_{3}$, or

\item $\left\vert a\right\vert =5$, and $a$ has two monochromes isomorphic to
$C_{5}$.
\end{enumerate}
\end{definition}

The three simple cliques are depicted in Figure \ref{Fg:SimpleGallai}.

For the sake of concise exposition in what follows, we shorten the phrase ``the
triangle with vertices $u_{0}$, $u_{1}$, and $u_{2}$'' to ``the triangle
$u_{0}u_{1}u_{2}$.''

\begin{proposition}
\label{Pr:SimFacCliq}A clique is simple iff it is a nonsingleton irreducible
Gallai $2$-clique.
\end{proposition}

\begin{proof}
\ Let $a$ be a nonsingleton irreducible Gallai $2$-clique. $a$ cannot have six
or more elements, for the most basic form of Ramsey's Theorem
(\cite{Ramsey:1930}) asserts that a $2$-clique with six vertices has a
monochromatic triangle. $a$ cannot have three elements, for identification of
the two vertices connected by the edge with minority color constitutes a
nontrivial homogenous partition.

\setlength{\unitlength}{1.2mm}
\begin{figure}[th]
\begin{center}
\input{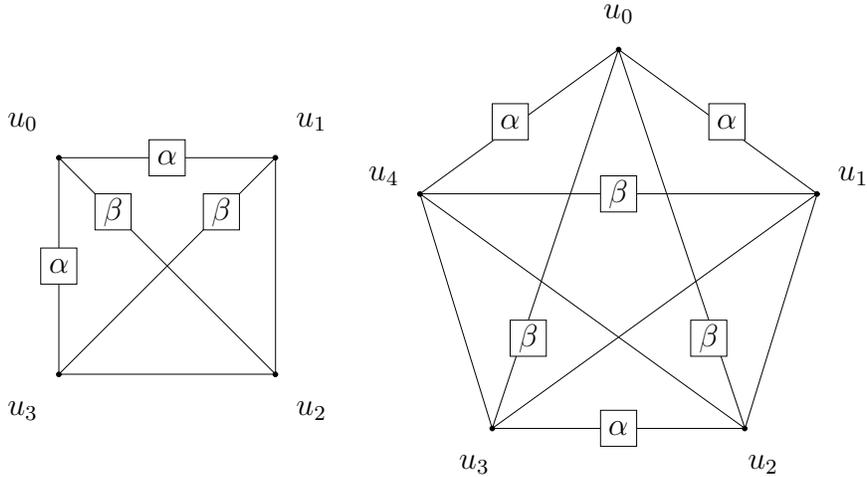}
\end{center}
\caption{Proving Proposition \ref{Pr:SimFacCliq}.}\label{Fg:45}
\end{figure}

Let $a=\{u_{0},u_{1},$ $u_{2},$ $u_{3}\}$. Without loss of generality
$\overline{u_{0}u_{1}}=\overline{u_{0}u_{3}}=\alpha$. If
$\overline{u_{0}u_{2}}=\alpha$ then $\left\{  u_{0},\left\{
u_{1},u_{2},u_{3}\right\} \right\}  $ is a nontrivial homogeneous partition,
hence $\overline{u_{0}u_{2}}=\beta\neq\alpha$. The triangle $u_{0}u_{1}u_{3}$
cannot be monochromatic, hence $\overline{u_{1}u_{3}}=\beta$. We are now in the
situation depicted in Figure \ref{Fg:45}, and it is easy to see that $a$ is
simple.

Let $a=\{u_{0},u_{1},$ $u_{2},$ $u_{3},$ $u_{4}\}$. Without loss of generality
$\overline{u_{0}u_{1}}=\overline{u_{0}u_{4}}=\alpha$, and $\overline
{u_{0}u_{3}}=\overline{u_{1}u_{4}}=\beta$. If $\overline{u_{0}u_{2}}=\alpha$
then $\overline{u_{1}u_{2}}=\beta$, but in that case any color assigned to
$u_{2}u_{4}$ would result in a colorful triangle. Thus
$\overline{u_{0}u_{2}}=\beta$, $\overline{u_{2}u_{3}}=\alpha$, and we are in
the situation depicted in Figure \ref{Fg:45}. It is then easy to see that $a$
is simple.
\end{proof}

\begin{theorem}
\label{Thm:ExGall}A clique is exact Gallai iff it is a hereditarily
irreducible $2$-clique with simple factors, such that higher factors use
different colors than lower factors.
\end{theorem}

\begin{proof}
Let $a$ by a hereditarily irreducible $2$-clique with simple factors. As long
as higher factors use different colors than lower factors, the argument given
in Proposition \ref{2} can be readily modified to show that every triangle has
edges of exactly two colors. Now consider an exact Gallai clique in a colored
graph. Apply Theorem \ref{4} to express it as a tree $2$-clique with
irreducible factors. Then these factors are exact by Proposition
\ref{Pr:SimFacCliq}, and clearly higher factors use differed colors than lower
factors, since otherwise a monochromatic triangle exists.
\end{proof}


\section{Full homomorphisms\label{Sec:FullHom}}

A \emph{full homomorphism} (\cite{HellNesetril:2004}) is a map $f:G\rightarrow
H$ between (uncolored) graphs such that for all $v_{i}\in V_{G}$,
\begin{displaymath}
    v_{1}v_{2}\in E_{G}\Longleftrightarrow f\left(  v_{1}\right)  f\left(
    v_{2}\right)  \in E_{H}.
\end{displaymath}
Note that the identity map is a full homomorphism, and that the composition of
full homomorphisms is itself a full homomorphism. Thus, graphs with full
homomorphisms constitute a category. For our purposes, however, we need only a
few simple properties of these maps, given in the following lemmas. In theses
lemmas and in what follows, we reserve the term embedding for the identity map
on an induced subgraph.

\begin{lemma}
\label{Lem:FullHom1}An embedding is a full homomorphism, and every full
homomorphism factors into a full surjection followed by an embedding. That is,
each full homomorphism $f:G\rightarrow H$ factors as $f=jf^{\prime}$,
\begin{displaymath}
    G\overset{f^{\prime}}{\rightarrow}f\left(  G\right)  \overset{j=\subseteq
    }{\rightarrow}H,
\end{displaymath}
where $f^{\prime}$ is the map $v\longmapsto f\left(  v\right)  $ onto the
induced subgraph with vertex set $f\left(  G\right)  $, and $j$ is the
embedding of this subgraph into $H$.
\end{lemma}

For a graph $G=(V,E)$, we can view $E$ as a relation on $V\times V$, and
therefore write $uEv$ in place of $uv\in E$, and $uE=\{v;\;uEv\}$.

A graph $G=\left(  V,E\right)$ is said to be \emph{reduced} if for all
$v_{i}\in V$,
\begin{displaymath}
    v_{1}E=v_{2}E\Longrightarrow v_{1}=v_{2}.
\end{displaymath}

Lemmas \ref{Lm:Folklore1}--\ref{Lm:Folklore2} are folklore:

\begin{lemma}\label{Lm:Folklore1}
\label{Lem:FullHom2}A full homomorphism out of a reduced graph is injective,
hence an embedding.
\end{lemma}

\begin{proof}
Let $f:G\rightarrow H$ be a full homomorphism, let $G$ be reduced, and let
$v_{i}\in V$ satisfy $f\left(  v_{1}\right)  =f\left(  v_{2}\right)  $. Since
for any $v\in V$ we have
\begin{displaymath}
    v_{1}v\in E_{G}\Longleftrightarrow f\left(  v_{1}\right)  f\left(  v\right)
    \in E_{H}\Longleftrightarrow f\left(  v_{2}\right)  f\left(  v\right)
    \Longleftrightarrow v_{2}v\in E_{G},
\end{displaymath}
it is clear that $v_{1}E_{G}=v_{2}E_{G}$. Because $G$ is reduced,
$v_{1}=v_{2}$.
\end{proof}

We wish to show that every graph has a reduced form. For that purpose, we fix a
graph $G=\left(  V_{G},E_{G}\right)  $ for the next few lemmas, and define
$\widehat{G}\equiv(V_{\widehat{G}},E_{\widehat{G}})$ by setting
\begin{displaymath}
    V_{\widehat{G}}\equiv\left\{  vE:v\in V_{G}\right\}  ,\;E_{\widehat{G}}
    \equiv\left\{  \left(  v_{1}E\right)  \left(  v_{2}E\right)  :v_{1}v_{2}\in
    E_{G}\right\}  .
\end{displaymath}
We first show that this definition makes sense.

\begin{lemma}
If $u_{1}E=u_{2}E$ and $v_{1}E=v_{2}E$ then
\begin{displaymath}
    u_{1}Ev_{1}\Longleftrightarrow u_{2}Ev_{2}.
\end{displaymath}

\end{lemma}

\begin{proof}
Since
\begin{displaymath}
    u_{1}Ev_{1}\Longleftrightarrow v_{1}\in u_{1}E=u_{2}E\Longleftrightarrow
    u_{2}\in v_{1}E=v_{2}E\Longleftrightarrow u_{2}Ev_{2},
\end{displaymath}
the result is clear.
\end{proof}

We define the canonical map $r_{G}:G\rightarrow\widehat{G}$ by the rule
$v\longmapsto vE$.

\begin{lemma}
\label{Lem:Ghat2}$\widehat{G}$ is reduced and $r_{G}$ is a full surjection.
Moreover, any function $h_{G}:\widehat{G}\rightarrow G$ which satisfies
$h_{G}\left( vE\right)  \in vE$ for all $v\in V_{G}$ constitutes a full
homomorphism such that $r_{G}h_{G}$ is the identity map on $\widehat{G}$.
\end{lemma}

The significance of $\widehat{G}$ is that it is the smallest full quotient of
$G$.

\begin{lemma}\label{Lm:Folklore2}
\label{Lem:Ghat3}$r_{G}$ is the smallest full surjection out of $G$. That is,
if $f:G\rightarrow H$ is a full surjection then there is a unique full
surjection $g:H\rightarrow\widehat{G}$ such that $gf=r_{G}$.
\end{lemma}

\begin{proof}
If $f\left(  v_{1}\right)  =f\left(  v_{2}\right)  $ then we claim that
$v_{1}E_{G}=v_{2}E_{G}$. For if $v\in v_{1}E_{G}$ then $f\left(  v_{1}\right)
E_{H}f\left(  v\right)  $, hence $f\left(  v_{2}\right)  E_{H}f\left(
v\right)  $ and $v_{2}E_{G}v$, and conversely. Thus we can define $h$ by
setting $h\left(  f\left(  v\right)  \right)  \equiv vE$. It is routine to
verify that $h$ has the properties claimed for it.
\end{proof}

It follows from Lemma \ref{Lem:Ghat3} that $G$ is reduced iff $r_{G}$ is an
isomorphism. We refer to $\widehat{G}$ as the \emph{reduced form of }$G$, and
we refer to the isomorphism type of $\widehat{G}$ as the \emph{type of }$G$.
Note that if $G$ is connected then so is its type.

Exact Gallai monochromes are characterized by their types.

\begin{proposition}
\label{Pr:ExGallMon}Monochromes of exact Gallai cliques are of type $P_{1}$,
$P_{3}$, or $C_{5}$, and every graph of one of these types appears as a
(spanning) monochrome in an exact Gallai clique.
\end{proposition}

\begin{proof}
According to Theorem \ref{Thm:ExGall}, we may think of an exact Gallai clique
as a tree $2$-clique $K\left(  T\right)  $ with simple factors. Let $G=\left(
V,E\right)  $ be a monochrome in $K\left(  T\right)  $, i.e., a component of
$K\left(  T\right)  \left(  \alpha\right)  $ for some color $\alpha$. Now
every edge of $E$ inherits its color from that of an edge connecting a sibling
pair in $S\left(  T\right)  $ as in Proposition \ref{5}, and all these sibling
pairs have a common parent $t$ because higher factors use different colors
than lower factors. Let $b=\left\{  t^{\prime}\in T:t^{\prime}\prec t\right\}
$ be the factor at $t$, and let $H$ be $b\left(  \alpha\right)  $. Let
$f:G\rightarrow H$ be the map which takes each $v\in V$ to its unique ancestor
in $b$. Then $f$ is clearly a full homomorphism, and since $b$ is simple, $H$
is isomorphic to $P_{1}$, $P_{3}$, or $C_{5}$.
\end{proof}

An induced subgraph of a reduced graph need not be reduced. The reduced induced
subgraphs of $C_{5}$ are $P_{1}$, $P_{3}$, and $C_{5}$ (the remaining $P_{2}$
reduces to $P_{1}$), the very graphs used to define simple cliques. This
observation permits a second characterization of exact Gallai types in
Corollary \ref{Cor:GallC5}, a result which uses the following trivial lemma.

\begin{lemma}
\label{Lem:Type}There exists a full homomorphism from $G$ into $H$ iff the
type of $G$ is embedded in the type of $H$.
\end{lemma}

\begin{proof}
In light of the $r_{X}:X \to \widehat{X}$ and $h_{X}: \widehat{X} \to X$ from
Lemma \ref{Lem:Ghat2}, there is an $f:G\to H$ iff there is a
$\widehat{f}:\widehat{G}\to\widehat{H}$. By Lemma \ref{Lem:FullHom2} the latter
is an embedding.
\end{proof}

\begin{corollary}
\label{Cor:GallC5}A connected graph is an exact Gallai monochrome iff it can
be mapped into $C_{5}$ by a full homomorphism.
\end{corollary}

$C_{5}$ reappears in a pivotal role in Section \ref{Sec:HomDual}.

Most questions about exact Gallai cliques can now be answered by
straightforward calculations. For example, we offer a concise proof of Theorem
1 of \cite{ChungGraham:1983}.

\begin{theorem}
\label{Thm:VerBnd}The largest number of vertices of an exact Gallai clique
colored by $k$ colors is $5^{\frac{k}{2}}$ if $k$ is even and $2\cdot
5^{\frac{k-1}{2}}$ if $k$ is odd.
\end{theorem}

\begin{proof}
Theorem \ref{Thm:ExGall} permits us to view an exact Gallai clique as a tree
$2$-clique with factors of size $n=2$, $4$ or $5$, in which higher factors use
different colors than lower factors. When $n=2$, the factor contributes one
color. When $n=4$ or $5$, the factor contributes two colors. The result
follows.
\end{proof}


\section{Homomorphism dualities\label{Sec:HomDual}}

\emph{In this section, all graphs are assumed to be connected.} Let
$\mathcal{M}$ be a class of graph homomorphisms. We write
\begin{displaymath}
    G\rightarrow_{\mathcal{M}}H
\end{displaymath}
to mean that there is a function $f:G\rightarrow H$ of $\mathcal{M}$.
Otherwise we write
\begin{displaymath}
    G\nrightarrow_{\mathcal{M}}H.
\end{displaymath}
Two sets $\mathcal{A}$ and $\mathcal{B}$ of graphs are said to be in a
\emph{homomorphism duality} (\cite{NesetrilPultr:1978}) if for every $G$
\begin{displaymath}
    \forall A\in\mathcal{A}\ \left(  A\nrightarrow_{\mathcal{M}}G\right)
    \Longleftrightarrow\exists B\in\mathcal{B}\ \left(
    G\rightarrow_{\mathcal{M}}B\right)  .
\end{displaymath}
In this section we take $\mathcal{M}$ to be the class of full homomorphisms.

\begin{theorem}
We have the homomorphism duality
\begin{displaymath}
    \left\{  C_{3},P_{4},A\right\}  \nrightarrow_{\mathcal{M}}G\quad
    \text{iff}\quad G\rightarrow_{\mathcal{M}}C_{5},
\end{displaymath}
and the connected graphs $G$ characterized by this condition are precisely the
mono\-chromes in exact Gallai cliques.
\end{theorem}

\begin{proof}
The condition displayed on the right characterizes the mo\-no\-chrom\-es in exact
Gallai cliques by Lemma \ref{Lem:Type} combined with Proposition
\ref{Pr:ExGallMon}. The same lemma also shows that the condition displayed on
the right implies the one on the left. Thus we have only to show that for any
connected graph $G$,
\begin{displaymath}
    \left\{  C_{3},P_{4},A\right\}  \nrightarrow_{\mathcal{M}}G\ \Rightarrow
    \ G\rightarrow_{\mathcal{M}}C_{5}.
\end{displaymath}

Suppose $G=(V,E)$ contains a copy of $C_{5}$, designated as in \ref{Note:SpGr}.
First observe that for every $v\in V$ there is an index $i$ for which
$vv_{i}\in E$. Indeed, if this were not the case then there would exist
vertices $u$ and $w$ and index $j$ such that $uw,wv_{j}\in E$ but $uv_{j}\notin
E$. (Consider the last three vertices on a shortest path from $v$ to $C_{5}$.)
In order to prevent $\left\{  u,w,v_{j},v_{j+1},v_{j+2}\right\}  $ and $\left\{
u,w,v_{j},v_{j-1},v_{j-2}\right\}  $ from being copies of $P_{4}$, we would
have to have $uv_{j+2}$, $uv_{j-2}\in E$, but then we would have a triangle.
Furthermore, $v$ cannot be connected with only one $v_{j}\in C_{5}$, or else
there would be a $P_{4}$-path $\left\{ w,v_{j},v_{j+1},v_{j+2},v_{j+3}\right\}
$. To avoid triangles, $v$ cannot be connected with two neighboring points
$v_{j}$, $v_{j+1}$ of $C_{5}$. Therefore, for every $v\in V$ there is exactly
one $i$, $0\leq i\leq 4$, such that $vv_{i-1}$, $vv_{i+1}\in E$; set $f(v)=i$.

We need to demonstrate that the map $v\longmapsto v_{f\left( v\right)  }$ is a
full homomorphism. If $uv\in E$ then we must have $f(u)=f(v)\pm1$, since
otherwise
\begin{displaymath}
    \left\{  v_{f(u)-1},v_{f(u)+1}\right\}  \cap\left\{  v_{f(v)-1},v_{f(v)+1}
    \right\}  \neq\emptyset,
\end{displaymath}
resulting in a triangle. Finally, if $f(u)=i$ and $f(v)=i+1$ then $uv\in E$
lest $\left\{  u,v_{i-1},v_{i-2},v_{i+2},v\right\}  $ be a copy of $P_{4}$.

Suppose $(V,E)$ does not contain a copy of $C_{5}$. Then the longest induced
path is a copy of $P_{k}$, $k=1$, $2$, or $3$, since
$P_{4}\nrightarrow_{\mathcal{M}}G$. Choose such a path in $G$, call it $P_{k}$,
and designate its vertices as in \ref{Note:SpGr}. Since
$P_{k}\rightarrow_{\mathcal{M}}C_{5}$, it suffices to construct a full
homomorphism $f:G\rightarrow P_{k}$.

If $k=1$ then $G$, by connectedness, is $P_{1}$ itself and the statement is
obvious. So suppose $k=2$, so that $P_{k}$ is $\left\{
v_{0},v_{1},v_{2}\right\}  $. Then for every $v\in V$ we have either $vv_{1}$
or $vv_{0}$ in $E$, and in the latter case we also have $vv_{2}$ in $E$, since
otherwise there would be a $P_{3}$-path. Set
\begin{displaymath}
    f(v)=\left\{
    \begin{array}
    [c]{lll}
    v_{1} & \text{if} & vv_{0},vv_{2}\in E\\
    v_{0} & \text{if} & vv_{1}\in E
    \end{array}
    \right.  .
\end{displaymath}
(Note that the range of $f$ is actually a $P_{1}$-path. This is not surprising,
for the reduced form of a $P_{2}$-path is a $P_{1}$-path, so that by Lemma
\ref{Lem:Type}, $G$ admits a full homomorphism into a $P_{2}$-path iff it
admits a full homomorphism into a $P_{1}$-path.) Now if $uv\in E$ then we could
not have $f(u)=f(v)=v_{i}$, for there would be the triangle $uvv_{1-i}$. And if
$f(u)f(v)$ is an edge, say $f(u)=v_{1}$ and $f(v)=v_{0}$, then, in order to
prevent $\left\{  u,v_{0},v_{1},v\right\}  $ from being a $P_{3}$-path, we have
to have $uv\in E$.

It remains only to handle the case in which $k=3$. We claim that each $v\in V$
has to be immediately connected with some $v_{i}\in P_{3}$. For otherwise
consider the last three points, call them $u$, $w$, and $v_{i}$, on a shortest
path connecting $v$ to $P_{3}$. Note that, since $u$ is not connected to
$P_{3}$, avoiding a $P_{4}$-path requires a second edge (other than $wv_{i}$)
joining $w$ to $P_{3}$, and there is precisely one such edge, else a triangle
arises.

If $i$ is $0$ then the possibilities for the second edge are $wv_{1}$,
$wv_{2}$, and $wv_{3}$, but these choices lead to a copy of $C_{3}$, a copy of
$A$, or a copy of $P_{4}$, respectively. If $i$ is $1$ then the possibilities
are $wv_{0}$, $wv_{2}$, and $wv_{3}$, but these choices lead to a copy of
$C_{3}$, a copy of $C_{3}$, or a copy of $A$, respectively. Symmetrical
arguments rule out the possibility that $i$ could be $2$ or $3$, and the claim
is proven.

Set
\begin{displaymath}
    f\left(  v\right)  \equiv\left\{
    \begin{array}
    [c]{lll}
    v_{0} & \text{if} & vv_{1}\in E\text{ and }vv_{3}\notin E\\
    v_{1} & \text{if} & vv_{0}\in E\text{ and }vv_{2}\in E\\
    v_{2} & \text{if} & vv_{1}\in E\text{ and }vv_{3}\in E\\
    v_{3} & \text{if} & vv_{2}\in E\text{ and }vv_{0}\notin E
    \end{array}
\right.  .
\end{displaymath}
The definition is correct, for if $vv_{0}\in E$ then $vv_{2}\in E$ to prevent
$\left\{  v,v_{0},v_{1},v_{2},v_{3}\right\}  $ from being either $P_{4}$ or
$C_{5}$, and similarly, if $vv_{3}\in E$ then also $vv_{1}\in E$. And the
value of the function at any argument is unique, since otherwise we would have
a copy of $C_{3}$. For the same reason, if $f(u)=f(v)$ then $uv\notin E$.

Now we must show that if $f\left(  u\right)  $ and $f\left(  v\right)  $ are
connected by an edge then so are $u$ and $v$. If $f(u)=v_{0}$ and $f(v)=v_{1}$
then $uv\in E$, since otherwise we would have an $A$-subgraph $\left\{
u,v_{1},v_{0},v_{2},v,v_{3}\right\}  $; likewise $f(u)=v_{2}$ and $f(v)=v_{3}$
imply $uv\in E$. If $f(u)=v_{1}$ and $f(v)=v_{2}$ then $uv\in E$, since
otherwise we would have a $P_{4}$-path $\left\{  u,v_{0},v_{1},v,v_{3} \right\}
$.

At last, we must show that if $f\left(  u\right)  $ and $f\left(  v\right)  $
are not connected by an edge then neither are $u$ and $v$.  If $f(u)=v_{0}$
and $f(v)=v_{2}$ then $uv\notin E$ because of the triangle $uvv_{1}$, and
similarly $uv\notin E$ if $f(u)=v_{1}$ and $f(v)=v_{3}$. Finally, if
$f(u)=v_{0}$ and $f(v)=v_{3}$ then $uv\notin E$ since otherwise we would have
an $A$-subgraph $\left\{  v_{0},v_{1},u,v,v_{2},v_{3}\right\}  $.
\end{proof}

\section{Acknowledgement}

We thank the anonymous referee for several useful comments and for the nice
proof of Proposition \ref{Pr:GallMon}.



\begin{thebibliography}{9}

\bibitem {Alexeev:2005} B. Alexeev, \emph{On lengths of rainbow cycles}, to
appear in Electronic J.~Combin.

\bibitem {CameronEdmonds:1997}K. Cameron and J. Edmonds, \emph{Lambda
composition}, J. Graph Theory \textbf{26} (1997), 9--16.

\bibitem {ChungGraham:1983}F. R. K. Chung and R. L. Graham, \emph{Edge-colored
complete graphs with precisely colored subgraphs}, Combinatorica \textbf{3}
(1983), 315--324.

\bibitem {Gallai:1967}T. Gallai, \emph{Transitiv oreintierbare Graphen}, Acta
Math. Acad. Sci. Hungar \textbf{18} (1967), 25--66. English translation by F.
Maffray and M. Preissmann, in Perfect Graphs, edited by J. L. Ramirez-Alfonsin
and B. A. Reed, John Wiley and Sons.

\bibitem {GyarfasSimonyi:2004}A. Gy\'{a}rf\'{a}s and G. Simonyi, \emph{Edge
colorings of complete graphs without tricolored triangles}, J. Graph Theory
\textbf{46} (2004), 211--216.

\bibitem {HellNesetril:2004}P. Hell and J. Ne\v{s}et\v{r}il, Graphs and
Homomorphisms, Oxford University Press, 2004.

\bibitem {Nesetril:1995}J. Ne\v{s}et\v{r}il, \emph{Ramsey theory}, in the
Handbook of Combinatorics, edited by R. Graham, M. Gr\"{o}tschel, and L.
Lov\'{a}sz, Elsevier, 1995.

\bibitem {NesetrilPultr:1978}J. Ne\v{s}et\v{r}il and A. Pultr, \emph{On
classes of relations and graphs determined by subobjects and factorobjects},
Discrete Math. \textbf{22} (1978), 287--300.

\bibitem {Ramsey:1930}F. P. Ramsey, \emph{On a problem for formal logic},
Proc. London Math. Soc. \textbf{30}, 264--286.
\end{thebibliography}
\end{document}